\documentclass[12pt]{amsart}
\usepackage{amsmath}
\usepackage{amsfonts}
\usepackage{amssymb}
\usepackage{graphicx}

\newtheorem{theo}{Theorem}[section]
\theoremstyle{plain}

\newtheorem{cor}[theo]{Corollary}

\newtheorem{dfn}[theo]{Definition}

\newtheorem{lemma}[theo]{Lemma}

\newtheorem{proposition}[theo]{Proposition}
\newtheorem{remark}[theo]{Remark}

\numberwithin{equation}{section}
                    
\begin{document}
\title{Lower Quasicontinuity, Joint Continuity and Related concepts}

\author{A. Bouziad}
\address{D\'epartement de Math\'ematiques, Universit\'e de Rouen,
UMR CNRS 6085, Avenue de l'Universit\'e, BP.12, F76801 Saint-Etienne 
du Rouvray, France}
\email{ahmed.bouziad@univ-rouen.fr}

\author{J.-P. Troallic}
\address{3 Impasse des Avocettes,
F76930 Cauville sur Mer,
France}
\email{jean-pierre.troallic@wanadoo.fr}



\begin{abstract}   Let   $X$ and $Y$ be  topological  spaces,  let  $Z$  be a metric space,   and  let $f: X\times Y\to Z$ 
be  a mapping.   It is shown that when $Y$ has a countable base $\mathcal B$, then under a rather general condition 
on the set-valued mappings $X\ni x\to f_x(B)\in 2^Z$, $B\in\mathcal B$,   there is a residual  set $R\subset X$ such 
that  for every  $(a,b)\in R\times Y$,  $f$ is jointly continuous at  $(a,b)$ if (and only if)  $f_a: Y\to Z$  is continuous at 
$b$.  Several new results are also established when the notion of continuity  is replaced by  that of quasicontinuity  
or  by  that of cliquishness.  Our approach allows  us to unify and improve various results from the literature.
\end{abstract}
\maketitle

\section{introduction} Let $f$ be a mapping of the product  $X\times Y$ of two    topological  spaces  into  a metric 
space $Z$.   Let  $\mathcal B$ be a countable collection  of  subsets of $Y$ and let $B$ be the set of all $y\in Y$ 
such that $\mathcal B$  includes  a   neighborhood  base  at  $y$ in $Y$.  In this note, we are mainly concerned
with the following question  (in the spirit of \cite{CT}):  find  some general assumptions   on the partial mappings 
$f_x$, $f^y$ ($x\in X$, $y\in Y$)  ensuring  the existence of a residual set $R\subset X$ such that $f$ is jointly 
continuous at each  point of $R\times B$.  In \cite{CT},  it is shown that if $f$ is separately  continuous,  then  such 
a set  $R$ exists. Theorem 3.3 below  is the main result of this note;  the first   of  assertions (1), (2) and (3) 
constituting  this theorem   is a  much more general  result than that  of  \cite{CT}  (although the proof is hardly more 
difficult);   the  other  two are  of the same sort as (1),  with the concept  of continuity  replaced  by that of 
quasicontinuity  in (2),  and that of  cliquishness  in (3).  \par

The paper is organized as follows.  In Section 2, we introduce and examine   the concept of lower quasicontinuity with
respect to the variable $x$ (at $(a, b)\in X\times Y$)   for  the mapping $f: X\times Y\to Z$. (Here, $Z$ may be any  
topological space.)   The statement and the proof of the main Theorem 3.3, as well as  some immediate  corollaries, 
occupy   Section 3.  Finally,   the results of Sections 2 and 3 are related  in Section 4  
to some well-known theorems from the literature.  \par

 As usual, $\mathbb N$, $\mathbb Q$ and $\mathbb R$   will denote, respectively,  the sets of natural, rational and 
real numbers. For any set  $A$,  $2^A$  will  denote  the set of all nonempty subsets of $A$.  For $x\in X$,  
${\mathcal V}_X(x)$ will denote  the set of all  neighborhoods of $x$ in $X$.

\section{Lower quasicontinuity and  mappings of two variables}

In this section, $X$, $Y$, $Z$ are topological spaces and $f: X\times Y\to Z$ is a mapping.  First, recall that a mapping 
$g:X\to Z$ is said to be {\it quasicontinuous}  at $a\in X$ if for each  neighborhood $U$ of $a$ in $X$ and each 
 neighborhood   $W$ of $g(a)$  in $Z$, there is an  open set  $O\subset X$  such that $\emptyset\not= O\subset U$ and 
$g(O)\subset W$ \cite{Th,Ma} (see also \cite{Ke,Bl}; the
terminology differs in \cite{Th,Bl}).  Obviously,  
$g$ is  said to be quasicontinuous  if it  is quasicontinuous  at each point  of $X$. (From now on,  obvious definitions will 
be omitted.)     A set-valued  mapping   $F: X\to 2^Z$  is said to be  {\it lower quasicontinuous}  at $x_0\in X$  if  for each 
 neighborhood $U$ of $x_0$ in $X$ and  each  open set $W\subset Z$ such that $F(x_0)\cap W\not=\emptyset$, there is 
an open set $O\subset X$ such that  $\emptyset\not= O\subset U$ and $F(x)\cap W\not=\emptyset$ for every 
$x\in O$  \cite{Ne,N} (and  \cite{BG} with another terminology). Given
any
$W\subset Z$, we put 
$F^-(W)=\{x\in X : F(x)\cap W\not=\emptyset\}$.   \par

The following concept  is formulated in \cite{MN}: The mapping $f:X\times Y\to Z$ is said to be {\it horizontally 
quasicontinuous}  at $(a, b)\in X\times Y$  if for each  neighborhood $U$ of $a$ in $X$,   each  neighborhood $V$ of $b$
 in $Y$, and  each neighborhood  $W$ of $f(a, b)$ in  $Z$, there is an  open set $O\subset X$ and $y\in V$ such that 
$\emptyset\not= O\subset U$  and $f(O\times\{y\})\subset W$.   Let us note  that if $f^{b}$ is quasicontinuous at $a$,  then 
$f$ is horizontally quasicontinuous at $(a,b)$.   It is also easy to check the following:

\begin{proposition} 
If  $a\in X$ and if $f: X\times Y\to Z$  is horizontally quasicontinuous at $(a, y)$ for every $y$  in a  nonempty  open  set 
$V\subset Y$, then the set-valued mapping  $F^V:X\ni x\to f_x(V)\in 2^Z$ is lower  quasicontinuous at $a$.
\end{proposition}

The converse  of Proposition 2.1 is false:  Let $f$ be the mapping  of ${\mathbb R}\times {\mathbb R}$ into the 
discrete space $\{0,1\}$ defined by $f(x,y)=1$ if and only if  $x-y\in\mathbb Q$;  then  for every nonempty open set  
$V\subset \mathbb R$, the  set-valued mapping   ${\mathbb R}\ni x\to f_x(V)\in 2^{\{0,1\}}$ is lower quasicontinuous; however, 
there is no point of ${\mathbb R}\times {\mathbb R}$  at which $f$ is  horizontally quasicontinuous. (Curiously enough,  both facts 
hold for the same  reason: For every $x\in\mathbb R$,  the sets  $x+\mathbb Q$ and $x+({\mathbb R}\setminus {\mathbb Q})$ 
are dense in $\mathbb R$.) 

\vskip 2mm

\begin{dfn} 
{\rm  We say that  $f: X\times Y\to Z$ is  {\it lower quasicontinuous with  respect to the variable $x$} 
({\it lower $X$-quasicontinuous}, for short) at the point $(a, b)\in X\times Y$ if for each  neighborhood $U$ of $a$ in $X$,  
each  neighborhood $V$ of $b$ in $Y$, and each neighborhood $W$ of $f(a,b)$ in  $Z$, there is an  open set $O\subset X$  such 
that   $\emptyset\not= O\subset U$  and $f(\{x\}\times V)\cap W\neq\emptyset$  for  every $x\in O$. Clearly,   ``$f: X\times Y\to Z$ 
 is  lower $X$-quasicontinuous''  means exactly that  for each nonempty open  set $V\subset Y$,   the set-valued mapping  
 $F^V:X\ni x\to f_x(V)\in 2^Z$ is lower  quasicontinuous.   On the other hand, let us point out  that  if  $f$ is horizontally 
quasicontinuous at $(a,b)$, then $f$  is   lower  quasicontinuous with respect to the variable $x$ at $(a,b)$. }
\end{dfn}

\vskip 2mm

The following proposition describes a wide class of mappings for which the results in Section 3 will apply (see Section 4).

\begin{proposition}
 Let  $V\subset Y$  be a  nonempty  open set, and suppose that   the mapping $f: X\times Y\to Z$ satisfies  one of the following:
\begin{itemize}
\item[{\rm ({\romannumeral 1})}] $f$ is vertically quasicontinuous at  every  point of $X\times V$,  and,  for each  $x\in X$,  there exists
 a dense subset $D_x$ of the space $V$ such that  $f$ is  lower $X$-quasicontinuous at every  point of $\{x\}\times D_x$.
\item[{\rm ({\romannumeral 2})}] $f$ is lower $Y$-quasicontinuous at every  point of $X\times V$,  and there exists a dense 
subset $D$ of  the space $V$ such that $f$ is  lower $X$-quasicontinuous at every  point of $X\times D$.
\end{itemize}
  Then  the set-valued mapping $F^V:X\ni x\to f_x(V)\in 2^Z$ is lower quasicontinuous.
\end{proposition}
\begin{proof}
Let $U\subset X$,  $W\subset Z$ be    nonempty open sets  such that $U$ meets $(F^V)^-(W)$, and let us 
show that there  is a nonempty open set  in $X$ contained in  $U$ and $(F^V)^-(W)$.   Let us choose an arbitrary point
 $(x_0,y_0)$ in $(U\times V)\cap f^{-1}(W)$. In case ({\romannumeral 1}), there is a nonempty open set
$V_0\subset V$ and $x_1\in U$ such that $f(\{x_1\}\times V_0)\subset W$; taking $y_1\in V_0\cap D_{x_1}$,
one can find a nonempty open set $U_0\subset U$ such that  $f(\{x\}\times V)\cap W\neq\emptyset$  for 
every $x\in U_0$. In case ({\romannumeral 2}), there is
a nonempty open set $V_1\subset V$ such that $f(U\times \{y\})\cap W\neq\emptyset$ for 
every $y\in V_1$.  Take  $y_2\in V_1\cap D$ and choose  $x_2\in U$ such that   $f(x_2,y_2)\in W$; then 
 there is a nonempty  open set $U_1\subset U$ such that $f(\{x\}\times V)\cap W\neq\emptyset$  for 
every $x\in U_1$. 
 \end{proof}

\vskip 2mm

The next statement in terms of the familiar  concept of quasicontinuity follows from  2.3; it will be 
used in Sections 3 and 4   to derive  several  Hahn or Kempisty type results  from  Theorem 3.3.

\begin{cor} Suppose  that  for each  $x\in X$,  $f_x$  is quasicontinuous  and there is a dense set $D_x\subset Y$
such that $f^y$ is quasicontinuous at $x$ for every $y\in D_x$.   Then 
  $f: X\times Y\to Z$   is lower $X$-quasicontinuous.
\end{cor}

\begin{remark}{\rm (1)  It is easy to see that the  mapping $f: X\times Y\to Z$ is lower quasicontinuous with
respect to the variable $x$   at   $(a,b)\in X\times Y$  if and only if    $f(a,b)\in \overline{f((U\cap A)\times V)}$
 for any  neighborhood $U$ of $a$ in $X$,   any  neighborhood $V$ of $b$ in $Y$,  and any dense subset $A$ of $X$.  \par

(2)  Let  $\tau_f$ be the topology on $X\times Y$  generated   by $f: X\times Y\to Z$  and   the two projections 
 $X\times Y\ni (x,y) \to x\in X$,  $X\times Y\ni (x,y) \to y\in Y$.  Then, as the proof of Proposition 2.3 shows, the set 
 $Q$ of all $(x,y)\in X\times Y$ at which $f$ is lower  $X$-quasicontinuous is closed with respect to $\tau_f$.  
 Thus, any condition ensuring that $Q$ is $\tau_f$-dense  in $X\times Y$ will imply that $f$  is lower 
$X$-quasicontinuous; for instance, the  conditions ({\romannumeral 1}) and ({\romannumeral 2}) 
in Proposition 2.3 (when satisfied for all nonempty open subsets  of $Y$)  are of this sort. }
\end{remark}

\section{joint continuity and related concepts}

In order to state our main result in 3.3 below, let us recall some  definitions concerning the concept of
quasicontinuity and that,  weaker,  of cliquishness.   Let $X$ and  $Y$ be two    topological  spaces  and let $(Z, d)$ 
be  a metric space.   A  mapping $g$ of   $X$  into   $(Z, d)$  is said to be  {\it cliquish} at   $a\in X$ if for 
each  $\varepsilon > 0$  and each  neighborhood $U$ of $a$  in $X$ there is an  open set $O\subset X$
  such that  $\emptyset\not= O\subset U$ and   ${\rm diam} (g(O))
\leq\varepsilon$ \cite{Th,Ma}. (If 
$W\in 2^Z$,  then  ${\rm diam} (W) = \sup\{ d(u, v) : u, v\in W\}$.)   A   mapping $f$ of  $X\times Y$ into    $(Z, d)$ 
 is said to be  {\it  quasicontinuous} ({\it cliquish})  {\it with respect to the variable} $x$ at  $(a,b)\in X\times Y$  if  
 for  each   neighborhood $V$ of $b$ in $Y$ and  each  $\varepsilon > 0$, there is a neighborhood  $U$ of $a$ in $X$ 
and an  open set $O\subset Y$  such that   $\emptyset\not= O\subset V$ and $d(f(a, b), f(x, y)) \leq \varepsilon$ for 
all  $x\in U$, $y\in O$  \cite{M,P1}  (respectively,  and
$d(f(a,y),f(x,y^{\prime})) \leq\varepsilon$ for all 
$x\in U$, $y, y^{\prime}\in O$ \cite{F}).    In relation to Theorem 4.4 below, it is worth noticing  that {\it  if  $a\in X$ 
and if  $f$ is cliquish with respect to the  variable $x$ at   $(a,y)\in X\times Y$  for  every $y\in Y$, then  the set 
of all $y\in Y$ such  that $f$ is continuous at $(a,y)$ is residual in $Y$} \cite{F}.   \par

Recall  also that  a collection $\mathcal B$ of nonempty open sets in  a topological space  is called a  {\it pseudobase} 
(or  {\it $\pi$-base}) for this space if  any  nonempty open set    contains some member of $\mathcal B$ \cite{O}.

\begin{lemma} 
Let $X$ be a  topological  space   and let $(Z, d)$ be a bounded  metric space.
Let $F$ be a lower quasicontinuous set-valued mapping of  $X$ into  $2^Z$. Let $z_0\in Z$.  Then the  
mapping $\phi : X\rightarrow \mathbb R$ defined by  $\phi(x) = \sup\{d(z_0, z) :
z\in F(x)\}$ {\rm (}$x\in X${\rm )} is  cliquish.
\end{lemma}
\begin{proof}   Let  $U$ be a nonempty open subset of the space $X$, and let $\varepsilon  > 0$. 
Let us put $r = \sup\{ \phi(x) :  x\in U\} - \varepsilon$ and let us choose $x_0\in U$  such that $r < \phi(x_0)$.
There is $z_1\in F(x_0)$ such that $r <  d(z_0, z_1)$; let us put $\rho = d(z_0, z_1) - r$. Let $O$ be  a nonempty open 
subset of $X$ contained in $U\cap F^-( B(z_1, \rho))$;  then $r < \phi(x) \leq r + \varepsilon$ for every $x\in O$.         
\end{proof}

For a set-valued mapping $F:X\to 2^Z$ and $U\subset X$, let $F(U)$  denote the set $\cup_{x\in U}F(x)$.
\begin{lemma} 
Let $X$ be a  topological  space   and let $(Z, d)$  be a  bounded metric space.
Let $F$ be a lower quasicontinuous set-valued mapping of $X$ into  $2^Z$.  Then 
$$A = \{x\in X :  \forall\varepsilon > 0, \exists U\in {\mathcal V}_X(x), {\rm diam} (F(U)) \leq 2 {\rm diam} (F(x)) + \varepsilon\}$$
is a residual subset of $X$.
\end{lemma}
\begin{proof}  Let $\varepsilon  > 0$, and let $A_{\varepsilon}$ be the union of all   open subsets $U$ of $X$ such
 that ${\rm diam} (F(U)) \leq 2{\rm diam} (F(x)) + \varepsilon$ for every $x\in U$.
Let us show that the open subset $A_{\varepsilon}$ of $X$ is dense in $X$;  
this will prove the lemma since  $\bigcap_{n\in{\mathbb N}} A_{{1}/{n + 1}}\subset A$.  
 Let $O$ be a nonempty open subset of $X$. Then, taking $x_0\in O$ and $z_0\in F(x_0)$, one can find a 
nonempty open subset $O^{\prime}$ of $X$ contained in $O\cap F^-(B(z_0,{\varepsilon}/4))$. 
 Let $\phi$ be the real-valued  mapping  on  $X$ defined  by  $\phi(x) = \sup\{d(z_0, z) \mid z\in F(x)\}$ for every $x\in X$. 
By Lemma 3.1,   there is a nonempty open set $O^{{\prime}{\prime}}\subset O^{\prime}$ 
such that  ${\rm diam} (\phi(O^{{\prime}{\prime}})) \leq {\varepsilon}/{4}$. Let us verify that  
$O^{{\prime}{\prime}}\subset A_{\varepsilon}$,   which will imply that $O\cap  A_{\varepsilon} \neq \emptyset$ 
since $O^{{\prime}{\prime}}\subset O$.  Let $x^{{\prime}{\prime}}\in O^{{\prime}{\prime}}$.  Let $z_i\in F(O^{{\prime}{\prime}})$, 
and let   $x_i\in O^{{\prime}{\prime}}$ such that $z_i\in F(x_i)$ ($i= 1, 2$); then
$$d(z_1, z_2) \leq d(z_1, z_0) + d(z_0, z_2) \leq\phi(x_1) + \phi(x_2) \leq 2\phi(x^{{\prime}{\prime}}) + {\varepsilon}/{2};$$
 consequently   ${\rm diam} (F(O^{{\prime}{\prime}})) \leq 2\phi(x^{{\prime}{\prime}}) + {\varepsilon}/{2}$.
Now, let  $z^{{\prime}{\prime}}\in F(x^{{\prime}{\prime}})\cap B(z_0, {\varepsilon}/{4})$; then
$$\sup\{d(z_0, z) : z\in F(x^{{\prime}{\prime}})\} 
\leq d(z_0, z^{{\prime}{\prime}}) + \sup\{ d(z^{{\prime}{\prime}}, z) : z\in F(x^{{\prime}{\prime}})\},$$
and consequently  $\phi(x^{{\prime}{\prime}})  \leq {\varepsilon}/{4} + {\rm diam} (F(x^{{\prime}{\prime}}))$. 
 Since  $O^{{\prime}{\prime}}$ is a nonempty open subset of  $X$ such that 
${\rm diam} (F(O^{{\prime}{\prime}})) \leq 2{\rm diam} (F(x^{{\prime}{\prime}})) + {\varepsilon}$ for every
$x^{{\prime}{\prime}}\in O^{{\prime}{\prime}}$, the inclusion  $O^{{\prime}{\prime}}\subset A_{\varepsilon}$ holds.  
 \end{proof}  

\vskip 2mm

\begin{theo} 
Let  $X$ and  $Y$  be  topological spaces,   let $(Z, d)$  be a metric space,  and   let $f : X\times Y\rightarrow  Z$  be a mapping.  
Let ${\mathcal B}$ be a countable collection of nonempty  subsets of $Y$. Let us suppose  that,  for every $V\in {\mathcal B}$,  
 the set-valued mapping $F^V:X\ni x\to f_x(V)\in 2^Z$ is lower quasicontinuous.  
Then  there is a residual set $R\subset X$ such that for each  $a\in R$:
\begin{itemize}
\item[{\rm (1)}]   If  $f_a$ is continuous at $b\in Y$  and if  $b$ has a neighborhood base contained in ${\mathcal B}$, 
then $f$ is continuous at $(a, b)$.
\item[{\rm (2)}]  If  $f_a$ is quasicontinuous at $b\in Y$ and   if some  neighborhood  of $b$ in $Y$ has a pseudobase contained in   
${\mathcal B}$,  then $f$  is  quasicontinuous with respect to the variable $x$ at $(a, b)$.
\item[{\rm (3)}]  If  $f_a$ is cliquish at $b\in Y$  and   if some  neighborhood  of $b$ in $Y$ has a pseudobase contained in   
${\mathcal B}$,  then   $f$  is cliquish with respect to the variable $x$ at $(a, b)$.
\end{itemize}
\end{theo}
\begin{proof} One can assume that $(Z, d)$ is bounded. By Lemma 3.2, for every $V\in{\mathcal B}$,  the set
$$R_V = \{x\in X : \forall\varepsilon > 0, \exists U\in {\mathcal V}_X(x), {\rm diam} (F^V(U)) \leq 2 {\rm diam} (F^V(x)) + \varepsilon\}$$
is a residual subset of $X$. Let us put $R = \bigcap_{V\in {\mathcal B}} R_V$;  since $R$ is the intersection of a countable
family of residual subsets of $X$, $R$ is a residual subset of $X$. Let us consider $a\in R$.  \par
(1) Let $b\in Y$ such that  $f_a$ is continuous at $b$ and  ${\mathcal B}$ contains  a neighborhood base at  $b$. 
Let $\varepsilon > 0$, and let ${V\in {\mathcal B}}$ be a neighborhood of $b$ in $Y$ such that  
 ${\rm diam} (f_a(V)) \leq {\varepsilon}/4$.  Since $a$ belongs to $R$, there  exists $U\in {\mathcal V}_X(a)$  
such that   ${\rm diam} (F^V(U)) \leq 2 {\rm diam} (F^V(a)) + {\varepsilon}/{2}$. 
Since $U\times V$ is a neighborhood of $(a, b)$  in $X\times Y$,  and since
$${\rm diam}(f(U\times V)) = {\rm diam} (F^V(U)) \leq  2 {\rm diam} (f_a(V)) + \varepsilon/2\leq    \varepsilon,$$
 $f$ is continuous at $(a, b)$.  \par
(2)  Let us suppose  that  $f_a$  is quasicontinuous at $b\in Y$, and let us suppose  that     ${\mathcal B}$  contains a pseudobase     
for some  neighborhood  $V^{\prime}$ of $b$ in $Y$.   Let $V\in{\mathcal V}_Y(b)$ and $\varepsilon > 0$.  
There is    a nonempty open set $O^{\prime}$ in  $Y$  such that $O^{\prime}\subset V\cap V^{\prime}$
and  $ f_a(O^{\prime})\subset B(f_a(b), {\varepsilon}/6)$.    Let   ${O\in {\mathcal B}}$   
contained in  $O^{\prime}$ and open  in $V^{\prime}$; $O$ is a nonempty  open set in $Y$ contained in $V$. 
  Since $a\in R$, there exists   $U\in {\mathcal V}_X(a)$  such that 
${\rm diam} (F^O(U)) \leq 2 {\rm diam} (F^O(a)) + {\varepsilon}/{6}$. Now, since  for any $(x, y)\in U\times O$,
$$d(f(a, b), f(x, y))\leq d(f_a(b), f_a(y)) + d(f_a(y), f_x(y))\leq  {\varepsilon}/{6} + {\rm diam} (F^O(U) \leq    \varepsilon,$$ 
statement  (2) is proved.  \par
(3)  Suppose that  $f_a$  is cliquish  at $b\in Y$  and  that     ${\mathcal B}$  contains a pseudobase     
for some  neighborhood  $V^{\prime}$ of $b$ in $Y$.   Let  $V\in{\mathcal V}_Y(b)$ and  $\varepsilon > 0$.  
 There is a nonempty open set $O^{\prime}$ in $Y$  such that   $O^{\prime}\subset  V\cap V^{\prime}$
and   ${\rm diam}( f_a(O^{\prime})) \leq {\varepsilon}/4$.   Let   ${O\in {\mathcal B}}$ 
contained in  $O^{\prime}$ and open in $V^{\prime}$;  $O$ is a nonempty  open set in $Y$ 
contained in $V$.  Since $a\in R$, there is $U\in {\mathcal V}_X(a)$
such that ${\rm diam} (F^O(U)) \leq 2 {\rm diam} (F^O(a)) + {\varepsilon}/{2}$. It follows from what preceeds that
$${\rm diam}(f(U\times O)) = {\rm diam} (F^O(U)) \leq   2 {\rm diam} (f_a(O)) + {\varepsilon}/{2}\leq    \varepsilon,$$      
which establishes  statement  (3).
\end{proof}

\vskip 2mm

\begin{remark} 
{\rm  In    point  (1)  of  Theorem 3.3,  the topology on   $Y$  is  quite  irrelevant    inasmuch as
the neighborhood of $b$ used  to establish the continuity of   $f$ at  $(a,b)$  belongs to the collection $\mathcal B$.  
 Similar  remarks  hold  for  point  (2)  and  point (3)  of   that  same theorem.}
\end{remark}

In view of    3.3  (and  2.4, for   Corollary 3.7), we can state:

\begin{cor} Let  $X$ and  $Y$ be  topological spaces,  let $Z$  be a metric space,  and  let $f: X\times Y\to Z$ be  a mapping. 
 Let us suppose that   $Y$ has a countable base and  that for every  $y\in Y$, the mapping $f^y$
is quasicontinuous. Then   there  is a residual set $R\subset X$ such that  for  every $(a,b)\in R\times Y$, 
$f$ is continuous  at $(a,b)$ provided that  $f_a$ is continuous  at $b$.
\end{cor}

\begin{cor} 
 Let  $X$ and  $Y$ be  topological spaces,  let $Z$  be a metric space,  and  let $f: X\times Y\to Z$ be  a mapping. 
Let us suppose that   $Y$ has a countable pseudobase and  that   for every  $y\in Y$, the mapping  $f^y$
is quasicontinuous. Then  there  is a residual set $R\subset X$ such that for every   $(a,b)\in R\times Y$, $f$  is 
quasicontinuous  {\rm (}cliquish{\rm )} with respect to the variable $x$  at
$(a,b)$ provided that  
$f_a$ is quasicontinuous  {\rm (}respectively,  cliquish{\rm )} at $b$.
\end{cor}

\begin{cor} Let  $X$ be  a topological space, let  $Y$  be  a topological space with  a countable base 
{\rm (}pseudobase{\rm )},   and  let $f: X\times Y\to Z$ be  a mapping of  $
X\times Y$ into a metric space $Z$. 
Suppose that  for each  $x\in X$,  the mapping $f_x$ is continuous 
{\rm (}quasicontinuous{\rm )} 
and there is a dense set $D_x\subset Y$  such that $f^y$ is quasicontinuous at $x$ for every $y\in D_x$. 
Then  there  is a residual set $R\subset X$ such that $f$ is   continuous
 {\rm (}respectively,  quasicontinuous with respect to the variable $x${\rm )}
at every point of $R\times Y$.
\end{cor}

\begin{remark}
 {\rm One can also obtain Theorem 3.3  by using the following lemma established 
in   \cite{Ku}:  {\it  Let $X$ be a topological space,   let $(Z, d)$ be a metric space,  and  let $F: X\to 2^Z$  be a   lower 
quasicontinuous set-valued mapping. Then for  every $r > 0$ there exists an open dense subset $U$ of $X$ 
and a continuous mapping $f: U\to Z$ such that $d(f(u), F(u)) < r$ for all $u\in U$. }}
\end{remark}

\section{New light  on some well-known results}
The aim of this section is to relate the results in Section 3 to some well-known results
from the literature.   In what follows,  $X$ is a Baire space,  $Y$ is a topological space,  and
$f$ is a mapping   of  $X\times Y$   into  a metric space $(Z, d)$.  \par

In \cite[Theorem 1]{Mi},  Mibu proves   that the set  of  continuity points of $f$ is residual in $X\times Y$ assuming 
   $f_x$    to be continuous for all $x\in X$,   $f^y$    continuous for all  $y$ in  a given dense subset  of $Y$,  and  
   $Y$    first countable.  Theorem 3.3  and  Corollary  2.4   above  immediately give  Mibu's result 
(in a somewhat  more general form).

 \begin{theo}  
 Suppose that   $Y$  is first countable   and  that,  for each  $x\in X$,  $f_x$  is   continuous  and 
$f^y$ is quasicontinuous at $x$ for every $y$ in a dense set $D_x\subset Y$.  Then  for each  $y\in Y$, there exists a 
residual set $R_y\subset X$ such that $f$ is continuous at every  point of $R_y\times \{y\}$.
\end{theo}

  In \cite[Theorem 3]{F}, Fudali concludes  that $f$ is cliquish assuming  $f_x$ to be cliquish for each  $x\in X$, 
 $f^y$   quasicontinuous for each $y\in Y$, and  $Y$   locally  second countable (a result stronger than
Mibu's Theorem  2 in  \cite{Mi}).  The following (slightly more general) statement is deduced from 3.3.  

\begin{theo} 
Suppose that $f^y$ is quasicontinuous for  every   $y\in Y$,  and suppose that there is a dense 
set $D\subset Y$  such  that for every $y\in D$:
\begin{itemize}
\item[{\rm ({\romannumeral 1})}]  There is a dense Baire subspace $Q_y$ of $X$ such 
 that $f_x$ is cliquish  at $y$ for every $x\in Q_y$, and 
\item[{\rm ({\romannumeral 2})}]  some neighborhood of $y$ in $Y$ has a countable pseudobase.
\end{itemize}
 Then  $f$ is cliquish.
\end{theo}
\begin{proof}
Let $U$  be a nonempty   open set in  $X$,  $V$  a nonempty   open set in  $Y$,  and  $\varepsilon>0$. 
Let us choose $b\in V\cap D$. By Theorem 3.3,   there is a residual set $R\subset X$ such that   for any $a\in R$, 
if  $f_a$ is cliquish  at $b$, then  $f$ is  cliquish  with respect to the variable $x$ at  $(a,b)$.
Since $Q_b$ is a dense Baire subspace  of $X$, $R\cap Q_b$ is dense in $X$; therefore  there is   $a\in U$  such that $f$ is 
cliquish with respect to the variable $x$ at $(a,b)$. Let us choose  $U_1\in {\mathcal V}_X(a)$  contained in $U$ 
and a nonempty open set $V_1$ in $Y$ contained in $V$ such  that $d(f(a,y),f(x,y'))<{\varepsilon}/{2}$  for 
all  $x\in U_1$ and $y,y'\in V_1$;  then  $d(f(x,y),f(x',y'))< \varepsilon$ for all  $x,x'\in U_1$  and $y,y'\in V_1$.
\end{proof} 

\vskip 2mm

 In \cite[Theorem 1]{M}, it is proved by Martin   that  if  $f_x$ is  quasicontinuous   for every $x\in X$,  
   $f^y$  quasicontinuous for every $y\in Y$, and  $Y$  second  countable,  then  $f$  is  quasicontinuous.  
Using 3.3 and 2.4 under condition (i),  and 3.3 under condition (ii),   Martin's result    can be improved as follows.

 \begin{theo} 
Suppose that for each $y\in Y$, some neighborhood of $y$ in $Y$ has a countable pseudobase,
 and suppose that  one of  the following holds:
 \begin{itemize} 
 \item[{\rm ({\romannumeral 1})}] For each  $x\in X$, $f_x$  is quasicontinuous  and there is a dense set $D_x\subset Y$
such that $f^y$ is quasicontinuous at $x$ for every $y\in D_x$. 
\item[{\rm ({\romannumeral 2})}] For each $y\in Y$, $f^y$ is quasicontinuous and there is a dense Baire subspace 
$Q_y\subset X$  such that $f_x$ is quasicontinuous at $y$ for every  $x\in Q_y$.
\end{itemize}
 Then $f$ is quasicontinuous.
\end{theo}
\begin{proof}  Let $(a,b)$ in $X\times Y$.  Let  $U$  be  an open neighborhood of  $a$ in  $X$,  $V$  an open 
neighborhood  of   $b$ in $Y$,  and  $\varepsilon>0$.   \par
Case ({\romannumeral 1}):   By  Theorem 3.3  and  Corollary  2.4, for each  $y\in Y$ there exists a residual set $R_y\subset X$ 
such  that  $f$ is quasicontinuous with respect to the  variable $x$   at every point of $R_y\times\{y\}$. 
Let $V_1\subset V$ be a nonempty open  set such that $d(f(a,y),f(a,b))<\varepsilon$ for every  $y\in V_1$. 
Choose $y_1\in V_1\cap D_a$ and let  $U_1\subset U$ be a nonempty open set such that $d(f(x,y_1),f(a,b))<\varepsilon$ 
 for every  $x\in U_1$.   Now, choosing $x_1\in U_1\cap R_{y_1}$ gives an open  neighborhood  $U_2$ of $x_1$ contained in 
 $U_1$   and a nonempty open set $V_2\subset V_1$ such that $d(f(x,y),f(a,b))<\varepsilon$ for every
$(x,y)\in U_2\times V_2$.  Hence  the mapping $f$ is quasicontinuous at $(a,b)$. \par
Case ({\romannumeral 2}):  By  Theorem 3.3,  for each  $y\in Y$ there is a residual set $R_y\subset X$ such that   for any 
$a^{\prime}\in R_y$,   if  $f_{a^{\prime}}$ is quasicontinuous   at $y$, then  $f$ is  quasicontinuous   with respect to the variable 
$x$ at  $(a^{\prime},y)$.   Choose a nonempty open set $U_1\subset U$   such that $d(f(x,b),f(a,b))<\varepsilon$ for each $x\in U_1$.
Since $Q_b$ is a dense Baire subspace  of $X$, $R_b\cap Q_b$ is dense in $X$.   Choosing  $x_1\in U_1\cap R_b\cap Q_b$ 
gives an open  neighborhood  $U_2$ of $x_1$    contained in  $U_1$  and a nonempty open set $V_1\subset V$ such that 
$d(f(x,y),f(a,b))<\varepsilon$ for every  $(x,y)\in U_2\times V_1$. Hence the mapping $f$ is quasicontinuous at $(a,b)$.      
\end{proof}

\vskip 2mm

 In  \cite[Theorem 4]{M},    Martin  also proves  that  if  $f_x$  is continuous  for every $x\in X$,   $f^y$  quasicontinuous for 
every  $y\in Y$,   and   $Y$    first  countable,  then  $f$  is  quasicontinuous with respect to the  variable $y$  
(cf. also   \cite{P1}).   The following  variant  of this theorem   is an easy  application of 3.3.  

\begin{theo}
 Let $b\in Y$ be a point  with a countable neighborhood base  in $Y$.  Suppose that $f^y$ is quasicontinuous for  every $y\in Y$, 
 and  suppose that  $f_x$ is continuous at $b$  for every  $x$ in a  given  dense Baire subspace $Q$ of $X$. Then: 
\begin{itemize}
\item[{\rm (1)}]  The mapping  $f$ is  quasicontinuous with respect to the variable   $y$ at each point of $X\times\{b\}$. 
\item[{\rm (2)}] The  set  of all $x\in X$ such that  the mapping $f$ is continuous at $(x, b)$ is  residual in $X$.
\end{itemize}
\end{theo}
\begin{proof} Let  $A$ be the set of all $x\in X$ such that   the mapping  $f$ is continuous at $(x, b)$.
By Theorem 3.3,   there is a residual set $R\subset X$ such that   $f$ is  continuous 
at each point $(x,b)\in R\times\{b\}$ provided that $f_x$ is continuous   at $b$. Remark  that  $Q$ being
 a dense Baire subspace  of $X$, the subset $R\cap Q$ of $A$ is dense in $X$, and consequently, $A$ is dense in $X$.  \par
(1) Let $a\in X$. Let  $U\in {\mathcal V}_X(a)$  and $\varepsilon>0$.   Let  $U^{\prime}$  a nonempty open set 
 in $X$ such that  $U^{\prime}\subset U$ and $f^b(U^{\prime})\subset B(f^b(a),  {\varepsilon}/{2})$; 
taking $x_0\in U^{\prime}\cap A$  gives   a nonempty open set   $O\subset U^{\prime}$ and    $V\in {\mathcal V}_Y(b)$ 
 such that, for every $(x,y)\in O\times V$,  $d(f(x,y),f(x_0,b))<{\varepsilon}/{2}$, and hence such that  $d(f(x,y),f(a,b))<{\varepsilon}$. \par
(2)   The $G_{\delta}$-set $A$ in $X$ being dense in $X$, it is  residual in $X$. 
\end{proof}

\vskip 2mm

As our last application of Theorem 3.3, we will now prove   a variant  of a recent result \cite[Theorem 2.5]{HH}.  Following \cite{HH}, 
a sequence $(f_n)_{n\in\mathbb N}$ of mappings of  $X$ into  $(Z, d)$  is said to be {\it  equi-quasicontinuous} at $x\in X$ if, for 
 each  neighborhood $U$ of $x$ in $X$ and each  $\varepsilon>0$,  there exists an  open set $O\subset X$ and  $n_0\in\mathbb N$
 such that $\emptyset\not= O\subset U$ and $d(f_n(x),f_n(y))<\varepsilon$ for every $n\geq n_0$ and $y\in O$.

 \begin{theo} 
Let $f_n : X\to Z$, $n\in\mathbb N$, be  a sequence   of quasicontinuous mappings, and let  $f: X\to Z$ be a mapping. Suppose
 that  for each $x\in X$,   the sequence $(f_n(x))_{n\in\mathbb N}$ clusters  to $f(x)$,  and suppose that for each $x$ in a given 
 dense Baire subspace $A$ of $ X$, the sequence $(f_n(x))_{n\in\mathbb N}$ converges to $f(x)$.  
Then,  for each $x\in A$, the following are equivalent:
 \begin{itemize}
 \item[{\rm (1)}] The sequence  $(f_n)_{n\in\mathbb N}$  is equi-quasicontinuous at $x$.
 \item[{\rm (2)}] 
 The mapping  $f: X\to Z$ is quasicontinuous at $x$.
 \end{itemize}
 \end{theo}
 \begin{proof} We refer to \cite{HH} for the implication $(1) \Rightarrow (2)$. To prove
 that (2) implies (1), let $Y={\mathbb N}\cup\{\infty\}$  equipped with the topology whose nonempty open sets are the subsets
$\{m\in{\mathbb N}:m\geq n\}\cup\{\infty\}$ of $Y$ ($n\in\mathbb N$).  Let $g: X\times Y\to Z$ be the mapping defined 
by $g(x,n)=f_n(x)$  and $g(x,\infty)=f(x)$.  It is easy to check that    for  each $n\in\mathbb N$,  the set-valued 
mapping $X\ni x\to \{f_m(x):m\geq n\}\cup\{f(x)\}\in 2^Z$   is lower quasicontinuous.  
 Let $a\in A$ such that $f$ is quasicontinuous at $a$. Let $U\in {\mathcal V}_X(a)$  and $\varepsilon>0$. 
There is $n_0\in \mathbb N$ such that  $d(f_n(a),f(a))<{\varepsilon}/3$ for every $n\geq n_0$.  Let  $V$ 
be a nonempty open set in $X$ such that $V\subset U$ and $d(f(a),f(x))<{\varepsilon}/3$ for 
every   $x\in V$. The mapping $g_x$ being continuous at $\infty$ for any $x\in A$, and $A\cap R$  being dense 
 in $X$ for any residual subset $R$ of $X$, it follows from  Theorem 3.3 that  there is $b\in V$ 
such  that $g$ is continuous  at $(b,\infty)$;   in particular,  there  is a nonempty  open set $O\subset V$ and $n_1\geq n_0$ 
such that $d(f(b),f_n(x))<{\varepsilon}/3$  for every $x\in O$ and $n\geq n_1$. For every $x\in O$ and $n\geq n_1$, we have
 $$d(f_n(x),f_n(a))\leq d(f_n(x),f(b))+d(f(b),f(a))+d(f(a),f_n(a))< \varepsilon.$$
 \end{proof}

\end{document}